\documentclass[12pt,leqno]{article}
\usepackage{graphicx} 
\usepackage{amsfonts}
\usepackage{graphicx, amsmath, amsthm, amsfonts, amssymb, color, ulem}
\usepackage{cite}
\usepackage{mathrsfs, float}
\newtheorem{thm}{Theorem}[section]
\newtheorem{ass}[thm]{Assumption}

\newtheorem{lem}[thm]{Lemma}

\theoremstyle{definition}

\definecolor{wco}{rgb}{0.5,0.2,0.3}
\renewcommand{\bar}{\overline}

\numberwithin{equation}{section}

\renewcommand{\tilde}{\widetilde}

\title{{\bf Wellposedness and averaging principle for conditional distribution dependent SDEs driven by standard Brownian motions and fractional Brownian motions\thanks{Supported by National Natural Science Foundation of China(12271524) and Jiangxi Provincial Natural Science Foundation(20224ACB201002).}}
}

\author
{ {\bf Li Tan}$^{\tt a, b}$\thanks{Contact E-mail address: tltanli@126.com}, {\bf Shengrong Wang}$^{\tt a, b}$
\\[0.5ex]
$^{\tt a}$School of Statistics and Data Science, Jiangxi University of \\ Finance and Economics, Nanchang, Jiangxi, 330013, China \\
$^{\tt b}$Key Laboratory of Data Science in Finance and Economics, Jiangxi \\ University of Finance and Economics, Nanchang, Jiangxi, 330013, China}

\begin{document}

\maketitle
\begin{abstract}
In this paper, we study a conditional distribution dependent stochastic differential equations driven by standard Brownian motion and fractional Brownian motion with Hurst exponent $H>\frac{1}{2}$ simultaneously. First, the existence and uniqueness of the equation is established by the fixed point theorem. Then, we show that the solutions of conditional distribution dependent stochastic differential equations can be approximated by the solutions of the associated averaged distribution dependent stochastic differential equations. \\
 {\it Keywords }: Conditional Mckean-Vlasov, Fractional Brownian motion, Wellposedness, Averaging principle
\end{abstract}

\section{Introduction}
Distribution dependent stochastic differential equations(SDEs), also called Mckean-Vlasov SDEs or mean-field SDEs, are a class of SDEs where the drift and diffusion coefficients depend not only on the current state of the process but also on its distribution. These equations arise naturally in the study of interacting particle systems with mean-field interactions and have found applications in various fields including statistical physics, mathematical finance, and engineering etc. A nonlinear Markov processes described by stochastic differential equations depending on their own distributions was first introduced by \cite{m96}. Recently, a lot of papers concentrate on the wellposedness, regularity properties, ergodicity, numerical methods and applications of distribution dependent SDEs.

Wellposedness is an essential condition while discussing SDEs. In classical SDE theory, if the coefficients of the equation satisfy the global Lipschitz and linear growth conditions with respect to the state variable, the existence and uniqueness of strong solutions can be proved by using the Picard iteration method or the fixed point theorem. Sznitman\cite{s91} first systematically established the connection between the McKean-Vlasov SDEs and the mean-field limit of the particle system, and estabilished the existence and uniqueness of solutions using the coupling method. Under general conditions, the existence and uniqueness can be obtained by constructing weak solutions to the equation through martingale problems or Girsanov transformations. In recent years, the focus of research has shifted to relaxing the Lipschitz assumption to cover more realistic application scenarios. For example: Mishura and Veretennikov\cite{mv20} discussed the existence and uniqueness of strong and weak solutions to the McKean-Vlasov SDEs under weaker regularity conditions, where they assume the coefficients no longer satisfy the linear growth condition with respect to the state variable; Mehri and Stannat\cite{ms19} investigated the existence and uniqueness of the McKean-Vlasov SDEs under Lyapunov-type conditions; De Raynal\cite{r20} considered the wellposedness of the solutions to degenerate McKean-Vlasov SDEs with H\"{o}lder continuous coefficients; Pascucci et al.\cite{pr24} investigated the existence and uniqueness of strong and weak solutions to degenerate McKean-Vlasov SDEs with rough coefficients; Huang and Wang\cite{hw19} gave a proof of the existence and uniqueness solutions to the McKean-Vlasov SDEs with non-degenerate noise, where the coefficients satisfy the integrability condition; Dieye et al. \cite{ddm22} investigated the existence of solutions to a class of McKean-Vlasov SDEs of integral-differential type with delay through analytic operator theory.

Most of the work on McKean-Vlasov SDEs focuses on the Brownian motion case. For Mckean-Vlasov SDEs driven by fractional Brownian motion, the conclusions are relatively scarce. Fan et al.\cite{fhs22} investigated the issue of the existence and uniqueness solutions to high-dimensional McKean-Vlasov stochastic differential equations driven by fractional Brownian motion. Shen et al.\cite{sxu22} constructed the McKean-Vlasov differential equations driven by fractional Brownian motion and standard Brownian motion, and established the existence and uniqueness theorem of solutions under the condition of superlinear growth by Carath\'{e}odory approximation Galeati et al.\cite{ghm23} gave the strong regularity of the solution to the McKean-Vlasov stochastic differential equation driven by additive fractional Brownian motion with irregular by stability estimates of singular stochastic differential equations driven by fractional Brownian motion.

Conditional McKean-Vlasov SDEs are those of SDEs where the drift and diffusion coefficients not only depends on the current state, but also the conditional distribution. In recent years, conditional McKean-Vlasov SDEs have been noticed by scholars since it can be used to describe the asymptotic behaviour for particle systems with mean-field interaction in a common environment. Erny\cite{e21} considerde the conditional McKean-Vlasov limits and conditional propagation of chaos property. Shao and Wei\cite{sw21} investigated a particle system with mean field interaction living in a random environment characterized by a regime-switching process, and studied the wellposedness and various properties of the limit conditional McKean-Vlasov SDEs. Shao\cite{stw24} investigated a conditional McKean-Vlasov SDE with jumps and Markovian regime-switching, and studied the wellposedness, propagation of chaos and averaging principle of the conditional McKean-Vlasov SDE.

The averaging principle, initiated by Khasminskii\cite{k68}, is an efficient tool in the theory of stochastic analysis and dynamical systems. However, as we know, the conditional McKean-Vlasov SDEs driven by standard Brownian motion and fractional Brownian motion have not been studied. In this paper, we aim at studying conditional distribution dependent SDEs driven by standard Brownian motion and fractional Brownian motion with Hurst exponent $H>\frac{1}{2}$ simultaneously. We will establish the existence and uniqueness of the equation by the fixed point theorem. Then, we will show that the solutions of conditional distribution dependent SDEs can be approximated by the solutions of the associated averaged distribution dependent stochastic differential equations.

The structure of the paper is as follows. In Section 2, a conditional distribution dependent SDE driven by fractional Brownian motion and standard Brownian motion is introduced. In Section 3, existence and uniqueness of the conditional distribution dependent SDE is derived by the fixed point theorem. In Section 4, the averaging principle for the conditional distribution dependent SDE is given.

\section{Preliminaries}
Let $(\Omega^0, \mathscr{F}^0, \mathscr{F}_t^0, \mathbb{P}^0)$ and $(\Omega^1, \mathscr{F}^1, \mathscr{F}_t^1, \mathbb{P}^1)$ be two filtered probability spaces such that $\mathscr{F}_{t+}^k:=\bigcap_{s>t}\mathscr{F}_{s}^k, k=0, 1$. Denote $\mathscr{F}$ as the completion of $\mathscr{F}^0\times\mathscr{F}^1$, $\mathscr{F}_t$ as the completion of $\mathscr{F}_t^0\times\mathscr{F}_t^1$ for $t\ge 0$, and $\mathbb{F}^0=(\mathscr{F}_t^0)_{t\ge0}$, $\mathbb{F}^1=(\mathscr{F}_t^1)_{t\ge0}$, $\mathbb{F}=(\mathscr{F}_t)_{t\ge0}$. Define $\Omega=\Omega^0\times\Omega^1$, $\mathbb{P}=\mathbb{P}^0\times\mathbb{P}^1$. For an element $w=(\omega^0,\omega^1)\in\Omega^0\times\Omega^1$, it means $w\in\Omega$. Further, we use $\mathbb{E}^0$ and $\mathbb{E}^1$ for taking expectation with respect to $\mathbb{P}^0$ and $\mathbb{P}^1$ respectively. Let $(\mathbb{R}^d,\langle\cdot,\cdot\rangle,|\cdot|)$ be the $d-$dimensional Euclidean space with inner product $\langle\cdot,\cdot\rangle$ and Euclidean norm $|\cdot|$. For a matrix, we denote by $\|\cdot\|$ the Euclidean norm. Set $\mathcal{C}([0,T];\mathbb{R}^d)$ be the family of all continuous functions from $[0,T]$ to $\mathbb{R}^d$, and $\mathcal{L}^2(\mathbb{R}^d)$ be the space of $\mathbb{R}^d$-valued random variables $\nu$ with $\sup_{t\in[0,T]}\mathbb{E}|\nu_t|^2<\infty$. Denote $\mathcal{P}(\mathbb{R}^d)$ by the family of all probability measures on $\mathbb{R}^d$, and define
\begin{equation*}
 \mathcal{P}_{\theta}(\mathbb{R}^d)=\bigg\{\mu\in\mathcal{P}(\mathbb{R}^d):\int_{\mathbb{R}^d}|x|^{\theta}\mu({\rm d}x)<\infty\bigg\}, \theta\ge1.
\end{equation*}
For $\mu\in\mathcal{P}_{\theta}(\mathbb{R}^d)$, define
\begin{equation*}
\mathcal{W}_{\theta}(\mu)=\left(\int_{\mathbb{R}^d}|x|^{\theta}\mu({\rm d}x)\right)^{1/\theta}.
\end{equation*}
For $\mu,\nu \in\mathcal{P}_{\theta}(\mathbb{R}^d)$, define the Wasserstein
distance between $\mu$ and $\nu$ as follows
\begin{equation*}\label{wasser}
 \mathbb{W}_{\theta}(\mu,\nu)=\bigg(\inf_{\pi\in\Pi(\mu,\nu)} \int_{\mathbb{R}^d\times\mathbb{R}^d}|x-y|^{\theta}\pi({\rm d}x,{\rm d}y)\bigg)^{1/\theta},
\end{equation*}
where $\Pi(\mu,\nu)$ is the set of probability measures on $\mathbb{R}^d\times\mathbb{R}^d$ with marginals $\mu$ and $\nu$.

For a random variable $X_t$, denote $\mathscr{L}(X_t)$ as the distribution of $X_t$, while we use $\mathscr{L}(X_t|\mathscr{F}_t^0)$ for the conditional distribution of $X_t$ given the $\sigma$-algebra $\mathscr{F}_t^0$. Consider the following $d$-dimensional conditional McKean-Vlasov SDE driven by fractional Browinan motion and standard Brownian motion
\begin{align}\label{MVSDE}
{\rm d}X_t=b(t, X_t, \mathscr{L}(X_t|\mathscr{F}_t^0)){\rm d}t+\sigma_W(t, X_t, \mathscr{L}(X_t|\mathscr{F}_t^0)){\rm d}W_t+\sigma_H(t, \mathscr{L}(X_t|\mathscr{F}_t^0)){\rm d}B_t^H,
\end{align}
where $W_t$ is an $r$-dimensional standard Brownian motion, and $B_t^H$ is an $m$-dimensional fractional Brownian motion with Hurst parameter $H\in(1/2,1)$ defined on $(\Omega^1, \mathscr{F}^1, \mathscr{F}_t^1, \mathbb{P}^1)$, $W_t$ and $B_t^H$ are independent of each other, and the initial data $X_0=\xi\in\mathcal{L}^2(\mathbb{R}^d)$ is a $d$-dimensional random variable defined on $\mathscr{F}^0$. Besides, $b: [0, T]\times\mathbb{R}^d\times\mathcal{P}_2(\mathbb{R}^d)\rightarrow\mathbb{R}^d$, $\sigma_W: [0, T]\times\mathbb{R}^d\times\mathcal{P}_2(\mathbb{R}^d)\rightarrow\mathbb{R}^d\otimes\mathbb{R}^r$, $\sigma_H: [0, T]\times\mathcal{P}_2(\mathbb{R}^d)\rightarrow\mathbb{R}^d\otimes\mathbb{R}^m$. We impose the following assumption:

\begin{ass}\label{ass1}
{\rm There exists a non-decreasing and bounded function $K_1(t)$ such that for any $t\in [0, T]$, and $x,y\in\mathbb{R}^d, \mu, \nu\in \mathcal{P}_{2}(\mathbb{R}^d)$
\begin{align*}
|b(t, x, \mu)-b(t, y, \nu)|^2+\|\sigma_W(t, x, \mu)-\sigma_W(t, y, \nu)\|^2\le K_1(t)(|x-y|^2+\mathbb{W}_{2}^2(\mu,\nu)),
\end{align*}
\begin{align*}
\|\sigma_H(t, \mu)-\sigma_H(t, \nu)\|^2\le K_1(t)\mathbb{W}_{2}^2(\mu,\nu),
\end{align*}
and
\begin{align*}
|b(t, 0, \delta_0)|^2+\|\sigma_W(t, 0, \delta_0)\|^2+\|\sigma_H(t, \delta_0)\|^2\le K_1(t),
\end{align*}
where $\delta_x(\cdot)$ stands for the Dirac delta measure concentrated at a point $x\in\mathbb{R}^d$.
 }
\end{ass}

\begin{lem}\label{lem2}
{\rm Under Assumption \ref{ass1}, we have
\begin{align*}
|b(t, x, \mu)|^2+\|\sigma_W(t, x, \mu)\|^2\le K_1(t)(1+|x|^2+\mathcal{W}_{2}^2(\mu)),
\end{align*}
\begin{align*}
\|\sigma_H(t, \mu)\|^2\le K_1(t)(1+\mathcal{W}_{2}^2(\mu))
\end{align*}
for any $t\in [0, T], x\in\mathbb{R}^d$, and $\mu\in \mathcal{P}_{2}(\mathbb{R}^d)$.}
\end{lem}
\begin{proof}
Since for $\mu\in \mathcal{P}_{2}(\mathbb{R}^d)$, one have $\mathbb{W}_{2}(\mu,\delta_0)=\mathcal{W}_{2}(\mu)$. By simple computation, it is easy to see that
the desired result follows.
\end{proof}

\section{Existence and Uniqueness}
\begin{thm}\label{th1}
{\rm Under Assumption \ref{ass1}, the conditional McKean-Vlasov SDE \eqref{MVSDE} admits a unique strong solution.}
\end{thm}
\begin{proof}
For $Y_t\in\mathcal{L}^2(\mathbb{R}^d)$, let $\mu_t=\mathscr{L}(Y_t|\mathscr{F}_t^0)$. Consider the following auxiliary $d$-dimensional distribution-independent SDE driven by fractional Browinan motion and standard Brownian motion
\begin{align}\label{SDE}
{\rm d}X_t^\mu=b(t, X_t^\mu, \mu_t){\rm d}t+\sigma_W(t, X_t^\mu, \mu_t){\rm d}W_t+\sigma_H(t, \mu_t){\rm d}B_t^H
\end{align}
with $X_t^\mu=X_0=\xi$. According to \cite[Theorem 2.2]{gn08}, under Assumption \ref{ass1}, \eqref{SDE} admits a unique strong solution $X_t^\mu$. Furthermore, it is easy to obtain that $\sup_{t\in[0,T]}\mathbb{E}|X_t^\mu|^2<\infty$. That is, $X^\mu\in\mathcal{L}^2(\mathbb{R}^d)$. Define an operator
\begin{align*}\label{oper}
\Psi: \mathcal{L}^2(\mathbb{R}^d)\rightarrow\mathcal{L}^2(\mathbb{R}^d), Y\mapsto X^\mu.
\end{align*}
Now, it is necessary to show that $\Psi$ is a strict contraction. For another $\tilde{Y}_t\in\mathcal{L}^2(\mathbb{R}^d)$, denote $\nu_t=\mathscr{L}(\tilde{Y}_t|\mathscr{F}_t^0)$. By Assumption \ref{ass1}, the H\"{o}lder inequality, the Burkholder-Davis-Gundy inequality, we get
\begin{equation}\label{diff}
\begin{aligned}
\mathbb{E}\left(\sup\limits_{s\in[0,t]}|X_s^\mu-X_s^\nu|^2\right)\le& 3\mathbb{E}\left(\sup\limits_{s\in[0,t]}\left|\int_0^s[b(r, X_r^\mu, \mu_r)-b(r, X_r^\nu, \nu_r)]{\rm d}r\right|^2\right)\\
&+3\mathbb{E}\left(\sup\limits_{s\in[0,t]}\left|\int_0^s[\sigma_W(r, X_r^\mu, \mu_r)-\sigma_W(r, X_r^\nu, \nu_r)]{\rm d}W_r\right|^2\right)\\
&+3\mathbb{E}\left(\sup\limits_{s\in[0,t]}\left|\int_0^s[\sigma_H(r, \mu_r)-\sigma_H(r, \nu_r)]{\rm d}B_r^H\right|^2\right)\\
\le& (3T+12)\mathbb{E}\int_0^tK_1(r)\left[|X_r^\mu-X_r^\nu|^2+\mathbb{W}_{2}^2(\mu_r,\nu_r)\right]{\rm d}r\\
&+C_{\kappa,H}T^{2H-1}\mathbb{E}\int_0^tK_1(r)\mathbb{W}_{2}^2(\mu_r,\nu_r){\rm d}r.
\end{aligned}
\end{equation}
where $C_{\kappa,H}$ is a positive constant depending on $\kappa, H$, and for the last term, we use the techniques of in \cite[Theorem 3.1]{FHS22}. In detail, denote $\int_r^s(s-u)^{-\kappa}(u-r)^{\kappa-1}{\rm d}u:=C(\kappa)$. Take some $\kappa$ such that $1-H<\kappa<1/2$, then by the stochastic Fubini theorem, the H\"{o}lder inequality, and combining with Assumption \ref{ass1}, we get
\begin{equation}\label{frac1}
\begin{aligned}
&\mathbb{E}\bigg(\sup_{s\in[0,t]}\bigg\lvert \int_{0}^s[\sigma_H(r, \mu_r)-\sigma_H(r, \nu_r)]{\rm d}B_{r}^H\bigg\rvert^{2}\bigg)\\
=&C(\kappa)^{-2}\mathbb{E}\bigg(\sup_{s\in[0,t]}\bigg\lvert \int_{0}^s\left(\int_r^s(s-u)^{-\kappa}(u-r)^{\kappa-1}{\rm d}u\right)[\sigma_H(r, \mu_r)-\sigma_H(r, \nu_r)]{\rm d}B_{r}^H\bigg\rvert^{2}\bigg)\\
=&C(\kappa)^{-2}\mathbb{E}\bigg(\sup_{s\in[0,t]}\bigg\lvert \int_{0}^s(s-u)^{-\kappa}\left(\int_0^u(u-r)^{\kappa-1}[\sigma_H(r, \mu_r)-\sigma_H(r, \nu_r)]{\rm d}B_{r}^H\right){\rm d}u\bigg\rvert^{2}\bigg)\\
\le&\frac{C(\kappa)^{-2}}{1-2\kappa}\mathbb{E}\bigg(\sup_{s\in[0,t]}s^{1-2\kappa} \int_{0}^s\bigg\lvert\int_0^u(u-r)^{\kappa-1}[\sigma_H(r, \mu_r)-\sigma_H(r, \nu_r)]{\rm d}B_{r}^H\bigg\rvert^{2}{\rm d}u\bigg)\\
\leq& \frac{C(\kappa)^{-{2}}}{1-2\kappa}t^{1-2\kappa}\int_{0}^{t}\mathbb{E}\bigg|\int_{0}^s(s-r)^{\kappa-1}[\sigma_H(r, \mu_r)-\sigma_H(r, \nu_r)]{\rm d}B_{r}^H\bigg|^{{2}}{\rm d}s.
\end{aligned}
\end{equation}
For each $s\in[0,t]$, $\int_{0}^s(s-r)^{\kappa-1}[\sigma_H(r, \mu_r)-\sigma_H(r, \nu_r)]{\rm d}B_{r}^H$ is a centered Gaussian random variable. Thus, there exists a positive constant $C_{H}$ such that
\begin{equation}\label{frac2}
\begin{aligned}
&\mathbb{E}\bigg|\int_{0}^s(s-r)^{\kappa-1}[\sigma_H(r, \mu_r)-\sigma_H(r, \nu_r)]{\rm d}B_{r}^H\bigg|^{{2}}\\
\le&\int_0^s\int_0^s(s-u)^{\kappa-1}\|\sigma_H(u, \mu_u)-\sigma_H(u, \nu_u)\|(s-v)^{\kappa-1}\|\sigma_H(v, \mu_v)-\sigma_H(v, \nu_v)\||u-v|^{2H-2}{\rm d}u{\rm d}v\\
\le&C_H\left(\int_0^s(s-r)^{\frac{\kappa-1}{H}}\|\sigma_H(r, \mu_r)-\sigma_H(r, \nu_r)\|^{\frac{1}{H}}{\rm d}r\right)^{2H}.
\end{aligned}
\end{equation}
Combine \eqref{frac1} and \eqref{frac2}, we get
\begin{equation}\label{frac}
\begin{aligned}
&\mathbb{E}\bigg(\sup_{s\in[0,t]}\bigg\lvert \int_{0}^s[\sigma_H(r, \mu_r)-\sigma_H(r, \nu_r)]{\rm d}B_{r}^H\bigg\rvert^{2}\bigg)\\
\le&C_{\kappa, H}t^{1-2\kappa}\int_0^t\left(\int_0^s(s-r)^{\frac{\kappa-1}{H}}\|\sigma_H(r, \mu_r)-\sigma_H(r, \nu_r)\|^{\frac{1}{H}}{\rm d}r\right)^{2H}{\rm d}s\\
\le&C_{\kappa, H}t^{1-2\kappa}\left(\int_0^t\|\sigma_H(r, \mu_r)-\sigma_H(r, \nu_r)\|^{\frac{2}{2\kappa-1+2H}}{\rm d}r\right)^{2\kappa-1+2H}\\
\le&C_{\kappa, H}t^{2H-1}\int_0^t\|\sigma_H(r, \mu_r)-\sigma_H(r, \nu_r)\|^{2}{\rm d}r\\
\le&C_{\kappa, H}t^{2H-1}\int_0^tK_1(r)\mathbb{W}_2^2(\mu_r,\nu_r){\rm d}r.
\end{aligned}
\end{equation}

Note that $K(t)$ is a non-decreasing function for any $t\in [0, T]$, and $\mu_t=\mathscr{L}(Y_t|\mathscr{F}_t^0)$, we can rewrite \eqref{diff} as follows:
\begin{align*}
\mathbb{E}\left(\sup\limits_{s\in[0,t]}|X_s^\mu-X_s^\nu|^2\right)\le& C_1\int_0^t\mathbb{E}\left(\sup\limits_{r\in[0,s]}|X_r^\mu-X_r^\nu|^2\right){\rm d}r+C_2\int_0^t\mathbb{E}\left(\sup\limits_{r\in[0,s]}|Y_r-\tilde{Y}_r|^2\right){\rm d}r.
\end{align*}
By the Gronwall inequality, we get
\begin{align}\label{fixed}
\mathbb{E}\left(\sup\limits_{s\in[0,t]}|X_s^\mu-X_s^\nu|^2\right)\le& C_2e^{C_1t}\int_0^t\mathbb{E}\left(\sup\limits_{r\in[0,s]}|Y_r-\tilde{Y}_r|^2\right){\rm d}s.
\end{align}
Define $\tilde{d}_{2,t}(X,\tilde{X})=\mathbb{E}\left(\sup\limits_{s\in[0,t]}|X_s^\mu-X_s^\nu|^2\right)$ and $\tilde{d}_{2,t}(Y,\tilde{Y})=\mathbb{E}\left(\sup\limits_{s\in[0,t]}|Y_s-\tilde{Y}_s|^2\right)$, \eqref{fixed} gives
\begin{align*}
\tilde{d}_{2,t}(\Psi(Y),\Psi(\tilde{Y}))\le C_2e^{C_1t}\int_0^t\tilde{d}_{2,s}(Y,\tilde{Y}){\rm d}s.
\end{align*}
We use $\Psi^n$ to stand for the $n$-th composition of the mapping $\Psi$ with itself. With the techniques of \cite[Theorem 2.4]{stw24}, by iteration, one get
\begin{align*}
\tilde{d}_{2,t}(\Psi^n(Y),\Psi^n(\tilde{Y}))\le C_2e^{C_1t}\int_0^t\tilde{d}_{2,s}(\Psi^{n-1}(Y),\Psi^{n-1}(\tilde{Y})){\rm d}s\le \frac{(C_2Te^{C_1T})^n}{n!}\tilde{d}_{2,t}(Y,\tilde{Y}).
\end{align*}
Since for any $a>0$, $\lim\limits_{n\rightarrow\infty}a^n/n!=0$, one can find some $n$ large enough such that $(C_2Te^{C_1T})^n<n!$, this means $\Psi^n$ is a strict contraction. Then, by the fixed point theorem, $\Psi$ admits a unique fixed point in $(\mathcal{L}^2(\mathbb{R}^d), \tilde{d})$. Finally, \eqref{MVSDE} admits a strong solution. This completes the proof.
\end{proof}

\begin{lem}\label{mvmoment}
{\rm Under Assumption \ref{ass1}, the conditional McKean-Vlasov SDE \eqref{MVSDE} satisfies
\begin{align*}
\mathbb{E}\left(\sup\limits_{0\le t\le T}|X_t|^2\right)\le C_3,
\end{align*}
where $C_3$ is a positive constant.
}
\end{lem}
\begin{proof}
Rewrite \eqref{MVSDE} to the integral form as follows:
\begin{equation*}
\begin{aligned}
X_t=&\xi+\int_0^tb(s, X_s, \mathscr{L}(X_s|\mathscr{F}_s^0)){\rm d}s\\
&+\int_0^t\sigma_W(s, X_s, \mathscr{L}(X_s|\mathscr{F}_s^0)){\rm d}W_s+\int_0^t\sigma_H(s, \mathscr{L}(X_s|\mathscr{F}_s^0)){\rm d}B_s^H.
\end{aligned}
\end{equation*}
Similarly with \eqref{frac}, by Assumption \ref{ass1} and the Burkholder-Davis-Gundy inequality, we get
\begin{equation*}
\begin{aligned}
&\mathbb{E}\left(\sup\limits_{0\le s\le t}|X_s|^2\right)\\
\le& 4\|\xi\|^2+4T\mathbb{E}\int_0^t\left|b(s, X_s, \mathscr{L}(X_s|\mathscr{F}_s^0))\right|^2{\rm d}s\\
&+16\mathbb{E}\int_0^t\left\|\sigma_W(s, X_s, \mathscr{L}(X_s|\mathscr{F}_s^0))\right\|^2{\rm d}s+4\mathbb{E}\left(\sup\limits_{0\le s\le t}\left|\int_0^s\sigma_H(r, \mathscr{L}(X_r|\mathscr{F}_r^0)){\rm d}B_r^H\right|^2\right)\\
\le& C(T)\left(1+\mathbb{E}\int_0^t[|X_s|^2+\mathcal{W}_{2}^2(\mathscr{L}(X_s|\mathscr{F}_s^0))]{\rm d}s\right)
\end{aligned}
\end{equation*}
Since for all $\omega^0\in\Omega^0$, one have $(\mathscr{L}(X_s|\mathscr{F}_s^0)(\omega^0))(\mathcal{D})=\mathbb{P}^1(X_s(\omega^0, \cdot)\in \mathcal{D})$ for $\mathcal{D}\in\mathscr{B}(\mathbb{R}^d)$, and by the definition of $\mathcal{W}_{2}(\mu)$, we get
\begin{equation}\label{conwass}
\begin{aligned}
\mathbb{E}(\mathcal{W}_{2}^2(\mathscr{L}(X_s|\mathscr{F}_s^0)))\le\mathbb{E}(\mathbb{E}^1(|X_s|^2))=\mathbb{E}|X_s|^2.
\end{aligned}
\end{equation}
By using the Gronwall inequality, the desired result follows.
\end{proof}

\section{Averaging Principle}
Next, we will study the averaging principle for the the following equation:
\begin{equation}\label{epsimv}
\begin{aligned}
{\rm d}X_t^{\epsilon}=&b\left(\frac{t}{\epsilon}, X_t^{\epsilon}, \mathscr{L}(X_t^{\epsilon}|\mathscr{F}_t^0)\right){\rm d}t+\sigma_W\left(\frac{t}{\epsilon}, X_t^{\epsilon}, \mathscr{L}(X_t^{\epsilon}|\mathscr{F}_t^0)\right){\rm d}W_t\\
&+\sigma_H\left(\frac{t}{\epsilon}, \mathscr{L}(X_t^{\epsilon}|\mathscr{F}_t^0)\right){\rm d}B_t^H
\end{aligned}
\end{equation}
with the initial data $X_0^{\epsilon}=X_0=\xi\in\mathcal{L}^2(\mathbb{R}^d)$, and $\epsilon$ is a positive parameter. Suppose for \eqref{epsimv}, Assumption \ref{ass1} is also satisfied. Then, by Theorem \ref{th1}, \eqref{epsimv} admits a unique strong solution $X_t^{\epsilon}, t\in[0,T]$.
Define the following averaged equation:
\begin{equation}\label{average}
{\rm d}\bar{X}_t=\bar{b}(\bar{X}_t, \mathscr{L}(\bar{X}_t)){\rm d}t+\bar{\sigma}_W(\bar{X}_t, \mathscr{L}(\bar{X}_t)){\rm d}W_t+\bar{\sigma}_H(\mathscr{L}(\bar{X}_t)){\rm d}B_t^H
\end{equation}
with the initial data $\bar{X}_0=X_0=\xi$. Here, $\bar{b}: \mathbb{R}^d\times\mathcal{P}_2(\mathbb{R}^d)\rightarrow\mathbb{R}^d$, $\bar{\sigma}_W: \mathbb{R}^d\times\mathcal{P}_2(\mathbb{R}^d)\rightarrow\mathbb{R}^d\otimes\mathbb{R}^r$, $\bar{\sigma}_H: \mathcal{P}_2(\mathbb{R}^d)\rightarrow\mathbb{R}^d\otimes\mathbb{R}^m$ are Borel measurable functions. Now, we are going to show that the solution $X_t^{\epsilon}$ could be approximated by the solution $\bar{X}_t$. Firstly, in order to ensure the well-posedness of \eqref{average}, we impose another assumption.

\begin{ass}\label{ass2}
{\rm There exists a bounded function $K_2: (0, \infty)\rightarrow(0, \infty)$ with $\lim\limits_{T\rightarrow\infty}K_2(T)=0$ such that for any $x\in\mathbb{R}^d, \mu\in \mathcal{P}_{2}(\mathbb{R}^d)$
\begin{equation*}
\begin{aligned}
&\sup\limits_{t\ge 0}\left|\frac{1}{T}\int_t^{t+T}[b(s, x, \mu)-\bar{b}(x, \mu)]{\rm d}s\right|^2+\sup\limits_{t\ge 0}\frac{1}{T}\int_t^{t+T}\|\sigma_W(s, x, \mu)-\bar{\sigma}_W(x, \mu)\|^2{\rm d}s\\
\le& K_2(T)(1+|x|^2+\mathbb{W}_{2}^2(\mu,\delta_0)),
\end{aligned}
\end{equation*}
and
\begin{align*}
\sup\limits_{t\ge 0}\frac{1}{T}\int_t^{t+T}\|\sigma_H(s, \mu)-\bar{\sigma}_H(\mu)\|^2{\rm d}s\le K_2(T)(1+\mathbb{W}_{2}^2(\mu,\delta_0)).
\end{align*}
 }
\end{ass}

\begin{lem}
{\rm Under Assumption \ref{ass1} and Assumption \ref{ass2}, the averaged McKean-Vlasov SDE \eqref{average} admits a unique strong solution. Moreover, the solution satisfies
\begin{align*}
\mathbb{E}\left(\sup\limits_{0\le t\le T}|\bar{X}_t|^2\right)\le C_4,
\end{align*}
where $C_4$ is a positive constant.
}
\end{lem}
\begin{proof}
For any $x, y\in\mathbb{R}^d, \mu,\nu\in \mathcal{P}_{2}(\mathbb{R}^d)$ and $T>0$, by Assumption \ref{ass1} and Assumption \ref{ass2}, we compute
\begin{equation*}
\begin{aligned}
&|\bar{b}(x, \mu)-\bar{b}(y, \nu)|^2\\
=&\left|\bar{b}(x, \mu)-\frac{1}{T}\int_0^Tb(s,x,\mu){\rm d}s+\frac{1}{T}\int_0^T[b(s,x,\mu)-b(s,y,\nu)]{\rm d}s+\frac{1}{T}\int_0^Tb(s,y,\nu){\rm d}s-\bar{b}(y, \nu)\right|^2\\
\le& 3\left|\frac{1}{T}\int_0^T[\bar{b}(x, \mu)-b(s, x, \mu)]{\rm d}s\right|^2+3\left|\frac{1}{T}\int_0^T[b(s, x, \mu)-b(s, y, \nu)]{\rm d}s\right|^2\\
&+3\left|\frac{1}{T}\int_0^T[b(s, y, \nu)-\bar{b}(y, \nu)]{\rm d}s\right|^2\\
\le& 3K_2(T)(2+|x|^2+|y|^2+\mathbb{W}_{2}^2(\mu,\delta_0)+\mathbb{W}_{2}^2(\nu,\delta_0))+3K_1(T)(|x-y|^2+\mathbb{W}_{2}^2(\mu,\nu)),
\end{aligned}
\end{equation*}
and
\begin{equation*}
\begin{aligned}
|\bar{b}(0, \delta_0)|^2\le& 2\left|\frac{1}{T}\int_0^T[\bar{b}(0, \delta_0)-b(s, 0, \delta_0)]{\rm d}s\right|^2+2\left|\frac{1}{T}\int_0^Tb(s, 0, \delta_0){\rm d}s\right|^2\le 2(K_2(T)+K_1(T)).
\end{aligned}
\end{equation*}
By taking $T\rightarrow\infty$, it is easy to see that, there exists a positive constant $K$ such that
\begin{equation}\label{barb}
\begin{aligned}
|\bar{b}(x, \mu)-\bar{b}(y, \nu)|^2\le K(|x-y|^2+\mathbb{W}_{2}^2(\mu,\nu)),~ ~ ~  |\bar{b}(0, \delta_0)|^2\le L
\end{aligned}
\end{equation}
since we have $\lim\limits_{T\rightarrow\infty}K_2(T)=0$ and $K_1$ is bounded. Similarly, by Assumption \ref{ass1} and Assumption \ref{ass2}, we have
\begin{equation*}
\begin{aligned}
&\|\bar{\sigma}_W(x, \mu)-\bar{\sigma}_W(y, \nu)\|^2\\
\le& 3K_2(T)(2+|x|^2+|y|^2+\mathbb{W}_{2}^2(\mu,\delta_0)+\mathbb{W}_{2}^2(\nu,\delta_0))+3K_1(T)(|x-y|^2+\mathbb{W}_{2}^2(\mu,\nu)),
\end{aligned}
\end{equation*}
\begin{equation*}
\begin{aligned}
\|\bar{\sigma}_H(\mu)-\bar{\sigma}_H(\nu)\|^2\le 3K_2(T)(2+\mathbb{W}_{2}^2(\mu,\delta_0)+\mathbb{W}_{2}^2(\nu,\delta_0))+3K_1(T)\mathbb{W}_{2}^2(\mu,\nu),
\end{aligned}
\end{equation*}
and
\begin{equation*}
\begin{aligned}
\|\bar{\sigma}_W(0, \delta_0)\|^2\le 2(K_2(T)+K_1(T)),~ ~ ~\|\bar{\sigma}_H(\delta_0)\|^2\le 2(K_2(T)+K_1(T)).
\end{aligned}
\end{equation*}
Then, by taking $T\rightarrow\infty$, we get
\begin{equation}\label{barw}
\begin{aligned}
\|\bar{\sigma}_W(x, \mu)-\bar{\sigma}_W(y, \nu)\|^2\le K(|x-y|^2+\mathbb{W}_{2}^2(\mu,\nu)),~ ~ ~  \|\bar{\sigma}_W(0, \delta_0)\|^2\le L,
\end{aligned}
\end{equation}
\begin{equation}\label{barh}
\begin{aligned}
\|\bar{\sigma}_H(\mu)-\bar{\sigma}_H(\nu)\|^2\le K\mathbb{W}_{2}^2(\mu,\nu),~ ~ ~  \|\bar{\sigma}_W(\delta_0)\|^2\le L.
\end{aligned}
\end{equation}
Similarly to the proof process of Theorem \ref{th1}, by \eqref{barb}-\eqref{barh}, we can show that there exists a unique strong solution $\bar{X}_t$ to \eqref{average}. Now we are going to show the finiteness of the moment. By \eqref{barb}-\eqref{barh} and the Burkholder-Davis-Gundy inequality, we get
\begin{equation}\label{barxt}
\begin{aligned}
&\mathbb{E}\left(\sup\limits_{0\le s\le t}|\bar{X}_s|^2\right)\\
\le& 4\|\xi\|^2+4T\mathbb{E}\int_0^t\left|\bar{b}(\bar{X}_s, \mathscr{L}(\bar{X}_s))\right|^2{\rm d}s+16\mathbb{E}\int_0^t\left\|\bar{\sigma}_W(\bar{X}_s, \mathscr{L}(\bar{X}_s))\right\|^2{\rm d}s\\
&+4\mathbb{E}\left(\sup\limits_{0\le s\le t}\left|\int_0^s\bar{\sigma}_H(\mathscr{L}(\bar{X}_r)){\rm d}B_r^H\right|^2\right)\\
\le& C(T)\left(1+\mathbb{E}\int_0^t[|\bar{X}_s|^2+\mathcal{W}_{2}^2(\mathscr{L}(\bar{X}_s))]{\rm d}s\right)
\end{aligned}
\end{equation}
Referring to the proof in \cite[Proposition 3.4]{DST19}, for the $\theta$-Wasserstein metric, we have
\begin{align}\label{wasserme}
\mathbb{W}_{\theta}^{2}(\mathscr{L}({X_t}),\delta_{0})=\mathcal{W}_{\theta}^2(\mathscr{L}({X_{t}}))\leq\mathbb{E}|X_{t}|^2.
\end{align}
Combining with \eqref{barxt} and \eqref{wasserme}, we get
\begin{equation*}
\begin{aligned}
&\mathbb{E}\left(\sup\limits_{0\le s\le t}|\bar{X}_s|^2\right)\le C(T)+C(T)\int_0^t\mathbb{E}\left(\sup\limits_{0\le r\le s}|\bar{X}_r|^2\right){\rm d}s
\end{aligned}
\end{equation*}
The Gronwall inequality then leads to the desired assertion.
\end{proof}

\begin{lem}\label{epsilon}
{\rm Under Assumption \ref{ass1}, \eqref{epsimv} admits a unique strong solution, and the solution satisfies
\begin{align}\label{xepsilon2}
\mathbb{E}\left(\sup\limits_{0\le t\le T}|X^{\epsilon}_t|^2\right)\le C_5^{\epsilon}.
\end{align}
Moreover, for $0\leq s\le t\leq T $ we have
\begin{align}\label{difepsilon}
\mathbb{E}|X^{\epsilon}_t-X^{\epsilon}_s|^2\le C_6^{\epsilon}(|t-s|+|t-s|^2+|t-s|^{2H})
\end{align}
where $C_5^{\epsilon}$ is a positive constant depending on $\epsilon$.
}
\end{lem}
\begin{proof}
Since by taking similar process of Theorem \ref{th1} and Lemma \ref{mvmoment}, the first part can be obtained. We now concentrate on \eqref{difepsilon}.
By \eqref{epsimv}, we get
\begin{align*}
X_{t}^{\epsilon}-X_{s}^{\epsilon}=&\int_{s}^tb\left(\frac{r}{\epsilon}, X_{r}^{\epsilon}, \mathcal{L}({X_{r}^{\epsilon}|\mathscr{F}_r^0})\right){\rm d}r+\int_{s}^t\sigma_W\left(\frac{r}{\epsilon}, X_{r}^{\epsilon}, \mathcal{L}({X_{r}^{\epsilon}|\mathscr{F}_r^0})\right){\rm d}W_t\\
&+\int_{s}^t\sigma_H\left(\frac{r}{\epsilon},\mathcal{L}({X_{r}^{\epsilon}|\mathscr{F}_r^0})\right){\rm d}B_{r}^H.
\end{align*}
By the elementary inequality, the H\"{o}lder inequality, the Burkhold-Davis-Gundy inequality, Lemma \ref{lem2}, \eqref{xepsilon2}, combining with the derivation process of \eqref{frac} and \eqref{conwass}, we get
\begin{align*}
\mathbb{E}|X_{t}^{\epsilon}-X_{s}^{\epsilon}|^2\leq &3\mathbb{E}\bigg\lvert \int_{s}^tb\left(\frac{r}{\epsilon}, X_{r}^{\epsilon},\mathcal{L}({X_{r}^{\epsilon}|\mathscr{F}_r^0})\right){\rm d}r\bigg\rvert^{2}\\
&+3\mathbb{E}\bigg\lvert \int_{s}^t\sigma_W\left(\frac{r}{\epsilon}, X_{r}^{\epsilon}, \mathcal{L}({X_{r}^{\epsilon}|\mathscr{F}_r^0})\right){\rm d}W_t\bigg\rvert^{2}+3\mathbb{E}\bigg\lvert \int_{s}^t\sigma_H\left(\frac{r}{\epsilon},\mathcal{L}({X_{r}^{\epsilon}|\mathscr{F}_r^0})\right){\rm d}B_{r}^H\bigg\rvert^{2}\\
\leq & [3(t-s)+12]\mathbb{E}\int_{s}^{t}K_1\left(\frac{r}{\epsilon}\right)\left(1+|X_{r}^{\epsilon}|^2+\mathcal{W}_{2}^2(\mathcal{L}({X_{r}^{\epsilon}|\mathscr{F}_r^0}))\right){\rm d}r\\
&+3C_{\kappa,H}(t-s)^{2H-1}\int_{s}^tK_1\left(\frac{r}{\epsilon}\right)\left(1+\mathcal{W}_{2}^2(\mathcal{L}({X_{r}^{\epsilon}|\mathscr{F}_r^0}))\right){\rm d}r\\
\leq &C_6^{\epsilon}(|t-s|+|t-s|^2+|t-s|^{2H}).
\end{align*}
This completes the proof.
\end{proof}

\begin{lem}\label{wassesti}
{\rm Under Assumption \ref{ass1}, we have
\begin{align*}
\mathbb{E}\left(\mathbb{W}_{2}^2(\mathscr{L}(X_t^{\epsilon}|\mathscr{F}_t^0), \mathscr{L}(\bar{X}_t))\right)\le\mathbb{E}|X_t^{\epsilon}-\bar{X}_t|^2,
\end{align*}
and
\begin{align*}
\mathbb{W}_{2}^2(\mathscr{L}(X_t^{\epsilon}|\mathscr{F}_t^0), \mathscr{L}(X_s^{\epsilon}|\mathscr{F}_s^0))\le\mathbb{E}^1|X_t^{\epsilon}-X_s^{\epsilon}|^2.
\end{align*}
}
\end{lem}
\begin{proof}
The main method comes from \cite[Lemma 3.2]{stw24}. Since for all $\omega^0\in\Omega^0$, one have $(\mathscr{L}(X_t^{\epsilon}|\mathscr{F}_t^0)(\omega^0))(\mathcal{D})=\mathbb{P}^1(X_t^{\epsilon}(\omega^0, \cdot)\in \mathcal{D})$ for $\mathcal{D}\in\mathscr{B}(\mathbb{R}^d)$. Besides, the mapping $\omega^0\rightarrow\mathscr{L}(X_t^{\epsilon}(\omega^0,\cdot))$ from $(\Omega^0, \mathscr{F}^0, \mathbb{P}^0)$ to $\mathcal{P}_2(\mathbb{R}^d)$ is almost surely well defined under $\mathbb{P}^0$, and this provides a conditional law of $X_t^{\epsilon}$ given $\mathscr{F}_t^0$. It is easy to see the process $\bar{X}_t$ does not depend on $\Omega^0$, so, under $\mathbb{P}^1$, the distribution of  $\bar{X}_t$ is $\mathscr{L}(\bar{X}_t)$. By the definition of Wasserstein distance, we get
\begin{align*}
\mathbb{W}_{2}^2(\mathscr{L}(X_t^{\epsilon}|\mathscr{F}_t^0)(\omega^0), \mathscr{L}(\bar{X}_t))\le\mathbb{E}^1|X_t^{\epsilon}(\omega^0,\cdot)-\bar{X}_t|^2,
\end{align*}
and
\begin{align*}
\mathbb{W}_{2}^2\left(\mathscr{L}(X_t^{\epsilon}|\mathscr{F}_t^0)(\omega^0), \mathscr{L}(X_s^{\epsilon}|\mathscr{F}_s^0)(\omega^0)\right)\le\mathbb{E}^1|X_t^{\epsilon}(\omega^0,\cdot)-X_s^{\epsilon}(\omega^0,\cdot)|^2.
\end{align*}
The desired result follows.
\end{proof}

\begin{lem}\label{diffb}
{\rm Under Assumption \ref{ass2}, we have
\begin{equation*}
\begin{aligned}
\lim\limits_{\epsilon\rightarrow 0}\mathbb{E}\left(\sup\limits_{0\le t\le T}\left|\int_0^t\left[b\left(\frac{r}{\epsilon}, X_r^{\epsilon}, \mathscr{L}(X_r^{\epsilon}|\mathscr{F}_r^0)\right)-\bar{b}(X_r^{\epsilon}, \mathscr{L}(X_r^{\epsilon}|\mathscr{F}_r^0))\right]{\rm d}r\right|^2\right)=0.
\end{aligned}
\end{equation*}
}
\end{lem}
\begin{proof}
In order to show the result, we first take a partition to $[0,T]$. Set $\Delta=\sqrt{\epsilon}$. For $k=0,1,\cdots,N-1$, let $t_k=k\Delta$ and $t_N=T$ where $t_N-t_{N-1}\le\Delta.$ It is easy to see that $T\le N\Delta<T+\Delta$. Then we compute
\begin{equation}\label{bb}
\begin{aligned}
&\left|\int_0^t\left[b\left(\frac{r}{\epsilon}, X_r^{\epsilon}, \mathscr{L}(X_r^{\epsilon}|\mathscr{F}_r^0)\right)-\bar{b}(X_r^{\epsilon}, \mathscr{L}(X_r^{\epsilon}|\mathscr{F}_r^0))\right]{\rm d}r\right|^2\\
\le&N\sum\limits_{i=0}^{N-2}\left|\int_{t_i}^{t_{i+1}}\left[b\left(\frac{r}{\epsilon}, X_r^{\epsilon}, \mathscr{L}(X_r^{\epsilon}|\mathscr{F}_r^0)\right)-\bar{b}(X_r^{\epsilon}, \mathscr{L}(X_r^{\epsilon}|\mathscr{F}_r^0))\right]{\rm d}r\right|^2\\
&+N\left|\int_{[\frac{t}{\Delta}]\Delta}^{t}\left[b\left(\frac{r}{\epsilon}, X_r^{\epsilon}, \mathscr{L}(X_r^{\epsilon}|\mathscr{F}_r^0)\right)-\bar{b}(X_r^{\epsilon}, \mathscr{L}(X_r^{\epsilon}|\mathscr{F}_r^0))\right]{\rm d}r\right|^2\\
:=&N\sum\limits_{i=0}^{N-2}L_i+NL_N,
\end{aligned}
\end{equation}
where $[\frac{t}{\Delta}]$ is the integer part of $\frac{t}{\Delta}$. By the H\"{o}lder inequality, Lemma \ref{lem2} and \eqref{barb}, we get
\begin{equation*}
\begin{aligned}
L_N=&\left|\int_{[\frac{t}{\Delta}]\Delta}^{t}\left[b\left(\frac{r}{\epsilon}, X_r^{\epsilon}, \mathscr{L}(X_r^{\epsilon}|\mathscr{F}_r^0)\right)-\bar{b}(X_r^{\epsilon}, \mathscr{L}(X_r^{\epsilon}|\mathscr{F}_r^0))\right]{\rm d}r\right|^2\\
\le&\left(t-[\frac{t}{\Delta}]\Delta\right)\int_{[\frac{t}{\Delta}]\Delta}^{t}\left|b\left(\frac{r}{\epsilon}, X_r^{\epsilon}, \mathscr{L}(X_r^{\epsilon}|\mathscr{F}_r^0)\right)-\bar{b}(X_r^{\epsilon}, \mathscr{L}(X_r^{\epsilon}|\mathscr{F}_r^0))\right|^2{\rm d}r\\
\le&2\left(t-[\frac{t}{\Delta}]\Delta\right)\int_{[\frac{t}{\Delta}]\Delta}^{t}\left(\left|b\left(\frac{r}{\epsilon}, X_r^{\epsilon}, \mathscr{L}(X_r^{\epsilon}|\mathscr{F}_r^0)\right)\right|^2+\left|\bar{b}(X_r^{\epsilon}, \mathscr{L}(X_r^{\epsilon}|\mathscr{F}_r^0))\right|^2\right){\rm d}r\\
\le&2\left(t-[\frac{t}{\Delta}]\Delta\right)\int_{[\frac{t}{\Delta}]\Delta}^{t}\left[\left(K_1\left(\frac{r}{\epsilon}\right)+K\right)\left(1+|X_r^{\epsilon}|^2+\mathcal{W}_2^2(\mathscr{L}(X_r^{\epsilon}|\mathscr{F}_r^0))\right)\right]{\rm d}r\\
\le&2\left(K_1\left(\frac{T}{\epsilon}\right)+K\right)\epsilon\left[1+\sup\limits_{0\le t\le T}|X_t^{\epsilon}|^2+\mathbb{E}\left(\sup\limits_{0\le t\le T}|X_t^{\epsilon}|^2\right)\right].
\end{aligned}
\end{equation*}
Then we use \eqref{xepsilon2} to get
\begin{equation}\label{ln}
\begin{aligned}
\mathbb{E}\left(\sup\limits_{0\le t\le T}L_N\right)\le&2(2K_{5,\epsilon}+1)\epsilon\left(K_1\left(\frac{T}{\epsilon}\right)+K\right).
\end{aligned}
\end{equation}
By Assumption \ref{ass1}, Assumption \ref{ass2}, Lemma \ref{wassesti} and \eqref{barb}, we have
\begin{equation*}
\begin{aligned}
L_i=&\left|\int_{t_i}^{t_{i+1}}\left[b\left(\frac{r}{\epsilon}, X_r^{\epsilon}, \mathscr{L}(X_r^{\epsilon}|\mathscr{F}_r^0)\right)-\bar{b}(X_r^{\epsilon}, \mathscr{L}(X_r^{\epsilon}|\mathscr{F}_r^0))\right]{\rm d}r\right|^2\\
\le&3\left|\int_{t_i}^{t_{i+1}}\left[b\left(\frac{r}{\epsilon}, X_r^{\epsilon}, \mathscr{L}(X_r^{\epsilon}|\mathscr{F}_r^0)\right)-b\left(\frac{r}{\epsilon}, X_{t_i}^{\epsilon}, \mathscr{L}(X_{t_i}^{\epsilon}|\mathscr{F}_{t_i}^0)\right)\right]{\rm d}r\right|^2\\
&+3\left|\int_{t_i}^{t_{i+1}}\left[b\left(\frac{r}{\epsilon}, X_{t_i}^{\epsilon}, \mathscr{L}(X_{t_i}^{\epsilon}|\mathscr{F}_{t_i}^0)\right)-\bar{b}\left(X_{t_i}^{\epsilon}, \mathscr{L}(X_{t_i}^{\epsilon}|\mathscr{F}_{t_i}^0)\right)\right]{\rm d}r\right|^2\\
&+3\left|\int_{t_i}^{t_{i+1}}\left[\bar{b}\left(X_{t_i}^{\epsilon}, \mathscr{L}(X_{t_i}^{\epsilon}|\mathscr{F}_{t_i}^0)\right)-\bar{b}(X_r^{\epsilon}, \mathscr{L}(X_r^{\epsilon}|\mathscr{F}_r^0))\right]{\rm d}r\right|^2\\
\le&3\sqrt{\epsilon} K_1\left(\frac{T}{\epsilon}\right)\int_{t_i}^{t_{i+1}}\left(|X_r^{\epsilon}-X_{t_i}^{\epsilon}|^2+\mathbb{W}^2_2(\mathscr{L}(X_r^{\epsilon}|\mathscr{F}_r^0), \mathscr{L}(X_{t_i}^{\epsilon}|\mathscr{F}_{t_i}^0))\right){\rm d}r\\
&+3\left|\epsilon\int_{t_i/\epsilon}^{t_{i+1}/\epsilon}\left[b\left(r, X_{t_i}^{\epsilon}, \mathscr{L}(X_{t_i}^{\epsilon}|\mathscr{F}_{t_i}^0)\right)-\bar{b}\left(X_{t_i}^{\epsilon}, \mathscr{L}(X_{t_i}^{\epsilon}|\mathscr{F}_{t_i}^0)\right)\right]{\rm d}r\right|^2\\
&+3\sqrt{\epsilon}K\int_{t_i}^{t_{i+1}}\left(|X_r^{\epsilon}-X_{t_i}^{\epsilon}|^2+\mathbb{W}^2_2(\mathscr{L}(X_r^{\epsilon}|\mathscr{F}_r^0), \mathscr{L}(X_{t_i}^{\epsilon}|\mathscr{F}_{t_i}^0))\right){\rm d}r\\
\le&3\sqrt{\epsilon} \left(K_1\left(\frac{T}{\epsilon}\right)+K\right)\int_{t_i}^{t_{i+1}}\left(|X_r^{\epsilon}-X_{t_i}^{\epsilon}|^2+\mathbb{E}|X_r^{\epsilon}-X_{t_i}^{\epsilon}|^2\right){\rm d}r\\
&+3\epsilon K_2\left(\frac{1}{\sqrt{\epsilon}}\right)\left[1+\sup\limits_{0\le t\le T}|X_{t}^{\epsilon}|^2+\mathbb{E}\left(\sup\limits_{0\le t\le T}|X_{t}^{\epsilon}|^2\right)\right].
\end{aligned}
\end{equation*}
By \eqref{xepsilon2} and \eqref{difepsilon}, we get
\begin{equation}\label{li}
\begin{aligned}
N\sum\limits_{i=0}^{N-2}\mathbb{E}\left(\sup\limits_{0\le t\le T}L_i\right)\le&3\epsilon N^2K_2\left(\frac{1}{\sqrt{\epsilon}}\right)(2K_{5,\epsilon}+1)\\
&+6C_6^{\epsilon}\epsilon N^2\left(K_1\left(\frac{T}{\epsilon}\right)+K\right)\left(\sqrt{\epsilon}+\epsilon+\epsilon^{H}\right).
\end{aligned}
\end{equation}
Taking \eqref{ln} and \eqref{li} into \eqref{bb}, combining with \eqref{difepsilon}, we see
\begin{equation*}
\begin{aligned}
&\mathbb{E}\left(\sup\limits_{0\le t\le T}\left|\int_0^t\left[b\left(\frac{r}{\epsilon}, X_r^{\epsilon}, \mathscr{L}(X_r^{\epsilon}|\mathscr{F}_r^0)\right)-\bar{b}(X_r^{\epsilon}, \mathscr{L}(X_r^{\epsilon}|\mathscr{F}_r^0))\right]{\rm d}r\right|^2\right)\\
\le&2(2K_{5,\epsilon}+1)\epsilon N\left(K_1\left(\frac{T}{\epsilon}\right)+K\right)+3\epsilon N^2K_2\left(\frac{1}{\sqrt{\epsilon}}\right)(2K_{5,\epsilon}+1)\\
&+6C_6^{\epsilon}\epsilon N^2\left(K_1\left(\frac{T}{\epsilon}\right)+K\right)\left(\sqrt{\epsilon}+\epsilon+\epsilon^{H}\right)\\
\le&2(2K_{5,\epsilon}+1)\sqrt{\epsilon}(T+\sqrt{\epsilon})\left(K_1\left(\frac{T}{\epsilon}\right)+K\right)+3(T+\sqrt{\epsilon})^2K_2\left(\frac{1}{\sqrt{\epsilon}}\right)(2K_{5,\epsilon}+1)\\
&+6C_6^{\epsilon}(T+\sqrt{\epsilon})^2\left(K_1\left(\frac{T}{\epsilon}\right)+K\right)\left(\sqrt{\epsilon}+\epsilon+\epsilon^{H}\right)\\
\rightarrow&0,\text {when} ~ ~ \epsilon\rightarrow0.
\end{aligned}
\end{equation*}
This completes the proof.
\end{proof}

\begin{lem}\label{diffwh}
{\rm Under Assumption \ref{ass2}, we have
\begin{equation}\label{diffw}
\begin{aligned}
\lim\limits_{\epsilon\rightarrow 0}\mathbb{E}\int_0^T\left\|\sigma_W\left(\frac{r}{\epsilon}, X_r^{\epsilon}, \mathscr{L}(X_r^{\epsilon}|\mathscr{F}_r^0)\right)-\bar{\sigma}_W(X_r^{\epsilon}, \mathscr{L}(X_r^{\epsilon}|\mathscr{F}_r^0))\right\|^2{\rm d}r=0,
\end{aligned}
\end{equation}
and
\begin{equation}\label{diffh}
\begin{aligned}
\lim\limits_{\epsilon\rightarrow 0}\mathbb{E}\int_0^T\left\|\sigma_H\left(\frac{r}{\epsilon}, \mathscr{L}(X_r^{\epsilon}|\mathscr{F}_r^0)\right)-\bar{\sigma}_H(\mathscr{L}(X_r^{\epsilon}|\mathscr{F}_r^0))\right\|^2{\rm d}r=0.
\end{aligned}
\end{equation}
}
\end{lem}
\begin{proof}
We only show the first part, and the second part can be derived similarly. Set $\Delta=\sqrt{\epsilon}$. For $k=0,1,\cdots,N-1$, let $t_k=k\Delta$ and $t_N=T$ where $t_N-t_{N-1}\le\Delta.$ It is easy to see that $T\le N\Delta<T+\Delta$. Then we compute
\begin{equation*}\label{ww}
\begin{aligned}
&\mathbb{E}\int_0^T\left\|\sigma_W\left(\frac{r}{\epsilon}, X_r^{\epsilon}, \mathscr{L}(X_r^{\epsilon}|\mathscr{F}_r^0)\right)-\bar{\sigma}_W(X_r^{\epsilon}, \mathscr{L}(X_r^{\epsilon}|\mathscr{F}_r^0))\right\|^2{\rm d}r\\
=&\sum\limits_{i=0}^{N-1}\mathbb{E}\int_{t_i}^{t_{i+1}}\left\|\sigma_W\left(\frac{r}{\epsilon}, X_r^{\epsilon}, \mathscr{L}(X_r^{\epsilon}|\mathscr{F}_r^0)\right)-\bar{\sigma}_W(X_r^{\epsilon}, \mathscr{L}(X_r^{\epsilon}|\mathscr{F}_r^0))\right\|^2{\rm d}r\\
:=&\sum\limits_{i=0}^{N-1}\mathbb{E}L'_i,
\end{aligned}
\end{equation*}
By Assumption \ref{ass1}, Assumption \ref{ass2}, Lemma \ref{wassesti} and \eqref{barw}, we have
\begin{equation*}
\begin{aligned}
L'_i=&\int_{t_i}^{t_{i+1}}\left\|\sigma_W\left(\frac{r}{\epsilon}, X_r^{\epsilon}, \mathscr{L}(X_r^{\epsilon}|\mathscr{F}_r^0)\right)-\bar{\sigma}_W(X_r^{\epsilon}, \mathscr{L}(X_r^{\epsilon}|\mathscr{F}_r^0))\right\|^2{\rm d}r\\
\le&3\int_{t_i}^{t_{i+1}}\|\sigma_W\left(\frac{r}{\epsilon}, X_r^{\epsilon}, \mathscr{L}(X_r^{\epsilon}|\mathscr{F}_r^0)\right)-\sigma_W\left(\frac{r}{\epsilon}, X_{t_i}^{\epsilon}, \mathscr{L}(X_{t_i}^{\epsilon}|\mathscr{F}_{t_i}^0)\right)\|^2{\rm d}r\\
&+3\int_{t_i}^{t_{i+1}}\|\sigma_W\left(\frac{r}{\epsilon}, X_{t_i}^{\epsilon}, \mathscr{L}(X_{t_i}^{\epsilon}|\mathscr{F}_{t_i}^0)\right)-\bar{\sigma}_W\left(X_{t_i}^{\epsilon}, \mathscr{L}(X_{t_i}^{\epsilon}|\mathscr{F}_{t_i}^0)\right)\|^2{\rm d}r\\
&+3\int_{t_i}^{t_{i+1}}\|\bar{\sigma}_W\left(X_{t_i}^{\epsilon}, \mathscr{L}(X_{t_i}^{\epsilon}|\mathscr{F}_{t_i}^0)\right)-\bar{\sigma}_W(X_r^{\epsilon}, \mathscr{L}(X_r^{\epsilon}|\mathscr{F}_r^0))\|^2{\rm d}r\\
\le&3K_1\left(\frac{T}{\epsilon}\right)\int_{t_i}^{t_{i+1}}\left(|X_r^{\epsilon}-X_{t_i}^{\epsilon}|^2+\mathbb{W}^2_2(\mathscr{L}(X_r^{\epsilon}|\mathscr{F}_r^0), \mathscr{L}(X_{t_i}^{\epsilon}|\mathscr{F}_{t_i}^0))\right){\rm d}r\\
&+3\epsilon\int_{t_i/\epsilon}^{t_{i+1}/\epsilon}\|\sigma_W\left(r, X_{t_i}^{\epsilon}, \mathscr{L}(X_{t_i}^{\epsilon}|\mathscr{F}_{t_i}^0)\right)-\bar{\sigma}_W\left(X_{t_i}^{\epsilon}, \mathscr{L}(X_{t_i}^{\epsilon}|\mathscr{F}_{t_i}^0)\right)\|^2{\rm d}r\\
&+3K\int_{t_i}^{t_{i+1}}\left(|X_r^{\epsilon}-X_{t_i}^{\epsilon}|^2+\mathbb{W}^2_2(\mathscr{L}(X_r^{\epsilon}|\mathscr{F}_r^0), \mathscr{L}(X_{t_i}^{\epsilon}|\mathscr{F}_{t_i}^0))\right){\rm d}r\\
\le&3\left(K_1\left(\frac{T}{\epsilon}\right)+K\right)\int_{t_i}^{t_{i+1}}\left(|X_r^{\epsilon}-X_{t_i}^{\epsilon}|^2+\mathbb{E}|X_r^{\epsilon}-X_{t_i}^{\epsilon}|^2\right){\rm d}r\\
&+3\sqrt{\epsilon}K_2\left(\frac{1}{\sqrt{\epsilon}}\right)\left[1+\sup\limits_{0\le t\le T}|X_{t}^{\epsilon}|^2+\mathbb{E}\left(\sup\limits_{0\le t\le T}|X_{t}^{\epsilon}|^2\right)\right].
\end{aligned}
\end{equation*}
Thus,
\begin{equation*}
\begin{aligned}
&\mathbb{E}\int_0^T\left\|\sigma_W\left(\frac{r}{\epsilon}, X_r^{\epsilon}, \mathscr{L}(X_r^{\epsilon}|\mathscr{F}_r^0)\right)-\bar{\sigma}_W(X_r^{\epsilon}, \mathscr{L}(X_r^{\epsilon}|\mathscr{F}_r^0))\right\|^2{\rm d}r\\
\le&6 \left(K_1\left(\frac{T}{\epsilon}\right)+K\right)N\sqrt{\epsilon}C_6^{\epsilon}\left(\sqrt{\epsilon}+\epsilon+\epsilon^{H}\right)
+3N\sqrt{\epsilon}K_2\left(\frac{1}{\sqrt{\epsilon}}\right)(2K_{5,\epsilon}+1)\\
\le&6\left(K_1\left(\frac{T}{\epsilon}\right)+K\right)(T+\sqrt{\epsilon})C_6^{\epsilon}\left(\sqrt{\epsilon}+\epsilon+\epsilon^{H}\right)
+3(T+\sqrt{\epsilon})K_2\left(\frac{1}{\sqrt{\epsilon}}\right)(2K_{5,\epsilon}+1)\\
\rightarrow&0,\text {when} ~ ~ \epsilon\rightarrow0.
\end{aligned}
\end{equation*}
This completes the proof.
\end{proof}

\begin{thm}
{\rm Under Assumption \ref{ass1} and Assumption \ref{ass2}, we have
\begin{align*}
\lim\limits_{\epsilon\rightarrow0}\mathbb{E}\left(\sup\limits_{0\le t\le T}|X^{\epsilon}_t-\bar{X}_t|^2\right)=0.
\end{align*}
}
\end{thm}
\begin{proof}
By \eqref{epsimv} and \eqref{average}, we compute
\begin{equation}\label{k}
\begin{aligned}
\mathbb{E}\left(\sup\limits_{0\le s\le t}|X^{\epsilon}_s-\bar{X}_s|^2\right)\le& 3\mathbb{E}\left(\sup\limits_{0\le s\le t}\left|\int_0^s\left[b\left(\frac{r}{\epsilon}, X_r^{\epsilon}, \mathscr{L}(X_r^{\epsilon}|\mathscr{F}_r^0)\right)-\bar{b}(\bar{X}_r, \mathscr{L}(\bar{X}_r))\right]{\rm d}r\right|^2\right)\\
&+3\mathbb{E}\left(\sup\limits_{0\le s\le t}\left|\int_0^s\left[\sigma_W\left(\frac{r}{\epsilon}, X_r^{\epsilon}, \mathscr{L}(X_r^{\epsilon}|\mathscr{F}_r^0)\right)-\bar{\sigma}_W(\bar{X}_r, \mathscr{L}(\bar{X}_r))\right]{\rm d}W_r\right|^2\right)\\
&+3\mathbb{E}\left(\sup\limits_{0\le s\le t}\left|\int_0^s\left[\sigma_H\left(\frac{r}{\epsilon}, \mathscr{L}(X_r^{\epsilon}|\mathscr{F}_r^0)\right)-\bar{\sigma}_H(\mathscr{L}(\bar{X}_r))\right]{\rm d}B_r^H\right|^2\right)\\
:=&3K_1+3K_2+3K_3.
\end{aligned}
\end{equation}
By the elementary inequality,
\begin{equation}\label{k1}
\begin{aligned}
K_1=&\mathbb{E}\left(\sup\limits_{0\le s\le t}\left|\int_0^s\left[b\left(\frac{r}{\epsilon}, X_r^{\epsilon}, \mathscr{L}(X_r^{\epsilon}|\mathscr{F}_r^0)\right)-\bar{b}(\bar{X}_r, \mathscr{L}(\bar{X}_r))\right]{\rm d}r\right|^2\right)\\
\le&2\mathbb{E}\left(\sup\limits_{0\le s\le t}\left|\int_0^s\left[b\left(\frac{r}{\epsilon}, X_r^{\epsilon}, \mathscr{L}(X_r^{\epsilon}|\mathscr{F}_r^0)\right)-\bar{b}(X_r^{\epsilon}, \mathscr{L}(X_r^{\epsilon}|\mathscr{F}_r^0))\right]{\rm d}r\right|^2\right)\\
&+2\mathbb{E}\left(\sup\limits_{0\le s\le t}\left|\int_0^s\left[\bar{b}(X_r^{\epsilon}, \mathscr{L}(X_r^{\epsilon}|\mathscr{F}_r^0))-\bar{b}(\bar{X}_r, \mathscr{L}(\bar{X}_r))\right]{\rm d}r\right|^2\right)\\
:=&K_{11}+K_{12}.
\end{aligned}
\end{equation}
By Assumption \ref{ass1}, Assumption \ref{ass2} and the H\"{o}lder inequality, we derive from \eqref{barb} and Lemma \ref{wassesti} that
\begin{equation}\label{k12}
\begin{aligned}
K_{12}=&2\mathbb{E}\left(\sup\limits_{0\le s\le t}\left|\int_0^s\left[\bar{b}(X_r^{\epsilon}, \mathscr{L}(X_r^{\epsilon}|\mathscr{F}_r^0))-\bar{b}(\bar{X}_r, \mathscr{L}(\bar{X}_r))\right]{\rm d}r\right|^2\right)\\
\le&2T\mathbb{E}\int_0^t|\bar{b}(X_r^{\epsilon}, \mathscr{L}(X_r^{\epsilon}))-\bar{b}(\bar{X}_r, \mathscr{L}(\bar{X}_r))|^2{\rm d}r\\
\le&2KT\mathbb{E}\int_0^t\left(|X_r^{\epsilon}-\bar{X}_r|^2+\mathbb{W}_2^2(\mathscr{L}(X_r^{\epsilon}|\mathscr{F}_r^0), \mathscr{L}(\bar{X}_r))\right){\rm d}r\\
\le&4KT\mathbb{E}\int_0^t|X_r^{\epsilon}-\bar{X}_r|^2{\rm d}r.
\end{aligned}
\end{equation}
By the Burkhold-Davis-Gundy inequality and the elementary inequality,
\begin{equation}\label{k2}
\begin{aligned}
K_2=&\mathbb{E}\left(\sup\limits_{0\le s\le t}\left|\int_0^s\left[\sigma_W\left(\frac{r}{\epsilon}, X_r^{\epsilon}, \mathscr{L}(X_r^{\epsilon}|\mathscr{F}_r^0)\right)-\bar{\sigma}_W(\bar{X}_r, \mathscr{L}(\bar{X}_r))\right]{\rm d}W_r\right|^2\right)\\
\le&4\mathbb{E}\int_0^t\left\|\sigma_W\left(\frac{r}{\epsilon}, X_r^{\epsilon}, \mathscr{L}(X_r^{\epsilon}|\mathscr{F}_r^0)\right)-\bar{\sigma}_W(\bar{X}_r, \mathscr{L}(\bar{X}_r))\right\|^2{\rm d}r\\
\le&8\mathbb{E}\int_0^t\left\|\sigma_W\left(\frac{r}{\epsilon}, X_r^{\epsilon}, \mathscr{L}(X_r^{\epsilon}|\mathscr{F}_r^0)\right)-\bar{\sigma}_W(X_r^{\epsilon}, \mathscr{L}(X_r^{\epsilon}|\mathscr{F}_r^0)\right\|^2{\rm d}r\\
&+8\mathbb{E}\int_0^t\left\|\bar{\sigma}_W(X_r^{\epsilon}, \mathscr{L}(X_r^{\epsilon}|\mathscr{F}_r^0)-\bar{\sigma}_W(\bar{X}_r, \mathscr{L}(\bar{X}_r))\right\|^2{\rm d}r\\
:=&K_{21}+K_{22}.
\end{aligned}
\end{equation}
By Assumption \ref{ass1} and Assumption \ref{ass2}, we derive from \eqref{barw} and Lemma \ref{wassesti} that
\begin{equation}\label{k22}
\begin{aligned}
K_{22}=&8\mathbb{E}\int_0^t\left\|\bar{\sigma}_W(X_r^{\epsilon}, \mathscr{L}(X_r^{\epsilon}|\mathscr{F}_r^0)-\bar{\sigma}_W(\bar{X}_r, \mathscr{L}(\bar{X}_r))\right\|^2{\rm d}r\\
\le&8K\mathbb{E}\int_0^t\left(|X_r^{\epsilon}-\bar{X}_r|^2+\mathbb{W}_2^2(\mathscr{L}(X_r^{\epsilon}|\mathscr{F}_r^0), \mathscr{L}(\bar{X}_r))\right){\rm d}r\\
\le&16K\mathbb{E}\int_0^t|X_r^{\epsilon}-\bar{X}_r|^2{\rm d}r.
\end{aligned}
\end{equation}
Again by the elementary inequality, similar to the derivation process of \eqref{frac}, we get
\begin{equation}\label{k3}
\begin{aligned}
K_3=&3\mathbb{E}\left(\sup\limits_{0\le s\le t}\left|\int_0^s\left[\sigma_H\left(\frac{r}{\epsilon}, \mathscr{L}(X_r^{\epsilon}|\mathscr{F}_r^0)\right)-\bar{\sigma}_H(\mathscr{L}(\bar{X}_r))\right]{\rm d}B_r^H\right|^2\right)\\
\le&C_{\kappa,H}t^{2H-1}\mathbb{E}\int_0^t\left\|\sigma_H\left(\frac{r}{\epsilon}, \mathscr{L}(X_r^{\epsilon}|\mathscr{F}_r^0)\right)-\bar{\sigma}_H(\mathscr{L}(\bar{X}_r))\right\|^2{\rm d}r\\
\le&C_{\kappa,H}t^{2H-1}\mathbb{E}\int_0^t\left\|\sigma_H\left(\frac{r}{\epsilon}, \mathscr{L}(X_r^{\epsilon}|\mathscr{F}_r^0)\right)-\bar{\sigma}_H(\mathscr{L}(X_r^{\epsilon}|\mathscr{F}_r^0))\right\|^2{\rm d}r\\
&+C_{\kappa,H}t^{2H-1}\mathbb{E}\int_0^t\left\|\bar{\sigma}_H(\mathscr{L}(X_r^{\epsilon}|\mathscr{F}_r^0))-\bar{\sigma}_H(\mathscr{L}(\bar{X}_r))\right\|^2{\rm d}r\\
:=&K_{31}+K_{32}.
\end{aligned}
\end{equation}
By Assumption \ref{ass1} and Assumption \ref{ass2}, we derive from \eqref{barh} and Lemma \ref{wassesti} that
\begin{equation}\label{k32}
\begin{aligned}
K_{32}=&C_{\kappa,H}t^{2H-1}\mathbb{E}\int_0^t\left\|\bar{\sigma}_H(\mathscr{L}(X_r^{\epsilon}|\mathscr{F}_r^0))-\bar{\sigma}_H(\mathscr{L}(\bar{X}_r))\right\|^2{\rm d}r\\
\le&C_{\kappa,H}t^{2H-1}\mathbb{E}\int_0^t\mathbb{W}_2^2(\mathscr{L}(X_r^{\epsilon}|\mathscr{F}_r^0), \mathscr{L}(\bar{X}_r)){\rm d}r\\
\le&C_{\kappa,H}t^{2H-1}\mathbb{E}\int_0^t|X_r^{\epsilon}-\bar{X}_r|^2{\rm d}r.
\end{aligned}
\end{equation}
Substituting \eqref{k1}-\eqref{k32} into \eqref{k} gives
\begin{equation*}
\begin{aligned}
&\mathbb{E}\left(\sup\limits_{0\le s\le t}|X^{\epsilon}_s-\bar{X}_s|^2\right)\\
\le& (4KT+16K+C_{\kappa,H}t^{2H-1})\mathbb{E}\int_0^t\mathbb{E}\left(\sup\limits_{0\le r\le s}|X_r^{\epsilon}-\bar{X}_r|^2\right){\rm d}s+3(K_{11}+K_{21}+K_{31}).
\end{aligned}
\end{equation*}
Since from Lemma \ref{diffb} and Lemma \ref{diffwh}, we have $\lim\limits_{\epsilon\rightarrow0}(K_{11}+K_{21}+K_{31})\rightarrow0$, thus the Gronwall inequality gives
\begin{equation*}
\begin{aligned}
\lim\limits_{\epsilon\rightarrow0}\mathbb{E}\left(\sup\limits_{0\le s\le t}|X^{\epsilon}_s-\bar{X}_s|^2\right)\rightarrow0.
\end{aligned}
\end{equation*}
This completes the proof.
\end{proof}


\end{document}